\documentclass[10 pt]{amsart}
\usepackage[margin= 1.5in]{geometry}
\usepackage{verbatim}
\usepackage{amssymb}
\usepackage{amsthm}
\usepackage{amsmath}
\usepackage[usenames,dvipsnames] {color}
\usepackage{graphicx}
\usepackage[normalem]{ulem}
\usepackage{tikz}
\usetikzlibrary{matrix}
\numberwithin{equation}{section}
\usepackage[mathscr]{eucal}
\usepackage{mathtools}
\usepackage{pgfplots}
\usepackage{stmaryrd}

\title[Free Actions on Suspensions and Joins]{\Large Free Actions on $C^*$-algebra Suspensions and Joins by Finite Cyclic Groups}
\author{Benjamin Passer}
\address{Technion-Israel Institute of Technology \\ Haifa, Israel }
\email{benjaminpas@tx.technion.ac.il}
\thanks{Supported in part by NSF DMS 1300280 and 1363250, and by a Zuckerman Fellowship at the Technion.}

\begin{document}


\def\polhk#1{\setbox0=\hbox{#1}{\ooalign{\hidewidth\lower1.5ex\hbox{`}\hidewidth\crcr\unhbox0}}}

\def\ba{\begin{aligned}}
\def\ea{\end{aligned}}
\def\be{\begin{equation}}
\def\ee{\end{equation}}
\def\beu{\begin{equation*}}
\def\eeu{\end{equation*}}

\def\C{\mathbb{C}}
\def\D{\mathbb{D}}
\def\R{\mathbb{R}}
\def\Rn{\mathbb{R}^n}
\def\S{\mathbb{S}}
\def\Sn{\mathbb{S}^n}
\def\Z{\mathbb{Z}}
\def\T{\mathbb{T}}
\def\N{\mathbb{N}}
\def\RP{\mathbb{RP}}
\def\Q{\mathbb{Q}}
\def\B{\mathbb{B}}

\def\F{\mathfrak{F}}
\def\del{\partial}

\def\!={\neq}
\def\l{\langle}
\def\r{\rangle}

\theoremstyle{definition}

\newtheorem{defn}[equation]{Definition}
\newtheorem{lem}[equation]{Lemma}
\newtheorem{prop}[equation]{Proposition}
\newtheorem{thm}[equation]{Theorem}
\newtheorem{claim}[equation]{Claim}
\newtheorem{ques}[equation]{Question}
\newtheorem{fact}[equation]{Fact}
\newtheorem{axiom}[equation]{Technical Axiom}
\newtheorem{newaxiom}[equation]{New Technical Axiom}
\newtheorem{cor}[equation]{Corollary}
\newtheorem{exam}[equation]{Example}
\newtheorem{conj}[equation]{Conjecture}
\newtheorem{altproof}[equation]{Alternative Proof}

\bibliographystyle{plain}
\maketitle

\begin{abstract}
We present a proof for certain cases of the noncommutative Borsuk-Ulam conjectures proposed by Baum, D\polhk{a}browski, and Hajac. When a unital $C^*$-algebra $A$ admits a free action of $\Z/k\Z$, $k \geq 2$, there is no equivariant map from $A$ to the $C^*$-algebraic join of $A$ and the compact \lq\lq quantum\rq\rq\hspace{0pt} group $C(\Z/k\Z)$. This also resolves D\polhk{a}browski's conjecture on unreduced suspensions of $C^*$-algebras. Finally, we formulate a different type of noncommutative join than the previous authors, which leads to additional open problems for finite cyclic group actions.
\end{abstract}

\vspace{.2 in}


\vspace{.2 in}


\vspace{.2 in}

\section{Introduction}\label{sec:intro}

In recent years, results on $C^*$-algebras have surfaced that generalize the Borsuk-Ulam theorem, which is an algebraic topological theorem on the spheres $\S^k = \{(x_1, \ldots, x_{k + 1}) \in \R^{k + 1}: x_1^2 + \ldots + x_{k + 1}^2 = 1 \}$. 

\begin{thm}[Borsuk-Ulam]
Any odd, continuous function from $\S^k$ to $\R^k$ must have the zero vector in its range, so there is no odd, continuous function from $\S^{k}$ to $\S^{k - 1}$. Equivalently, every odd, continuous function from $\S^{k - 1}$ to itself must be homotopically nontrivial. 
\end{thm}

The Borsuk-Ulam theorem has a multitude of generalizations in a wide variety of settings. The $C^*$-algebraic variants are motivated by the Gelfand-Naimark correspondence, which states that every commutative $C^*$-algebra $A$ with unit may be written uniquely as $C(X)$ for some compact Hausdorff space $X$. This correspondence takes the form of a contravariant functor, where continuous maps $X \to Y$ are dual to unital $*$-homomorphisms $C(Y) \to C(X)$. In this sense, one may think of a \textit{noncommutative} $C^*$-algebra with unit as a dual object to some nebulous \lq\lq compact noncommutative space.\rq\rq \hspace{1pt}This type of reasoning is more grounded when a noncommutative $C^*$-algebra $A$ is obtained from some $C(X)$ by a deformation procedure, such as the $q$-deformed (\cite{va90}) and $\theta$-deformed (\cite{ma91}, \cite{co02}, \cite{na97}) noncommutative spheres. In both of these cases, a $C^*$-algebraic analogue of the Borsuk-Ulam theorem holds, where oddness of functions is replaced with discussion of a natural $\Z/2\Z$ action on the algebra, which we call the antipodal action.

\begin{thm}(Yamashita, \cite{ya13} Corollary 15)
Fix $0 < q \leq 1$ and $m, n \in \Z^+$ with $m < n$. There is no antipodally equivariant, unital $*$-homomorphism from $C(\mathbb{S}^m_q)$ to $C(\mathbb{S}^n_q)$.
\end{thm}

\begin{thm}(P, \cite{pa15a} Corollary 4.6)
Fix $k \in \Z^+$ and two $\theta$-deformed spheres $C(\mathbb{S}^k_\phi)$ and $C(\mathbb{S}^{k+1}_\psi)$ of dimensions $k$ and $k + 1$, respectively. There is no antipodally equivariant, unital $*$-homomorphism from $C(\mathbb{S}^k_\phi)$ to $C(\mathbb{S}^{k+1}_\psi)$.
\end{thm}

The above theorems are related to, but they are not quite encapsulated by, Question 4 of A. Taghavi in \cite{ta12} and Conjecture 3.1 of L. D\polhk{a}browski in \cite{da15}. If $A$ is a unital $C^*$-algebra with a $\Z/2\Z$ action generated by $\alpha: A \to A$, then this action extends to the unreduced suspension 
\begin{equation*}
\Sigma A = \{ f \in C([-1, 1], A): f(-1), f(1) \in \C \}
\end{equation*}
as $\alpha_\Sigma$, defined by the rule
\be\label{eq:newaction}
\alpha_\Sigma f(t) = \alpha(f(-t)).
\ee
Note that we use the interval $[-1, 1]$ instead of $[0, 1]$, so our action looks different from the formulation in \cite{da15}, but it is equivalent.

The commutative sphere $C(\S^{k + 1})$ is the unreduced suspension of $C(\S^k)$, and the extension of the antipodal action is again antipodal. Now, for a $\theta$-deformed sphere $C(\S^k_\rho)$, $\Sigma C(\S^k_\rho)$ is another $\theta$-deformed sphere of dimension $k + 1$, but not all $\theta$-deformed spheres may be obtained in this way; the new generator $f(t) = t$ is always central. D\polhk{a}browski conjectured the following in pursuit of noncommutative Brouwer fixed point theorems.

\begin{conj}[D\polhk{a}browski, \cite{da15} Conjecture 3.1]\label{conj:dabgen}
For a unital $C^*$-algebra $A$ with a free action of $\Z/2\Z$, there is no $\Z/2\Z$-equivariant $*$-homomorphism $\phi: A \to \Sigma A$.
\end{conj}

This conjecture is deeply related to the quantum Borsuk-Ulam conjectures posed in \cite{ba15} for compact quantum groups acting on unital $C^*$-algebras, and these conjectures include a topological subcase. The underlying operations (see \cite{ba13} and \cite{da15b}) are the topological join
\beu
X * G \cong [0,1] \times X \times G / \sim
\eeu
\beu
(0, x_0, g) \sim (0, x_1, g) \textrm{ for all } x_0, x_1 \in X, g \in G
\eeu
\beu
(1, x, g_0) \sim (1, x, g_1) \textrm{ for all } x \in X, g_0, g_1 \in G
\eeu
of compact spaces $X$ and $G$, and the equivariant join
\beu
A \circledast_\delta H = \{f \in C([0,1], A \otimes_\textrm{min} H): f(0) \in \C \otimes H, f(1) \in \delta(A)\}
\eeu
of a unital $C^*$-algebra $A$ and a compact quantum group $(H, \Delta)$. Here $\delta: A \to A \otimes_\textrm{min} H$ denotes a coaction of $H$ on $A$, meaning $\delta$ is an injective, unital $*$-homomorphism satisfying the coassociativity identity
\beu
(\delta \otimes \mathrm{id}) \circ \delta = (\mathrm{id} \otimes \Delta) \circ \delta
\eeu
and the counitality condition
\beu
\overline{\left\{\sum_{\textrm{finite}} \delta(a_i)(1\otimes h_i): a_i \in A, h_i \in H \right\}} = A \otimes_{\mathrm{min}} H.
\eeu
The coaction $\delta$ is called free exactly when it meets the Ellwood condition
\be\label{eq:ellwood}
\overline{\left\{ \sum_\textrm{finite} (a_i \otimes 1)\delta(b_i): a_i, b_i \in A \right\} } = A \otimes_\mathrm{min} H
\ee
from Theorem 2.4 of \cite{el00}.

When a compact group $G$ acts on $X$ freely, the diagonal action of $G$ on $X*G$, $(t, x, g) \cdot h = (t, x\cdot h, gh)$, is also free. However, the induced (free) action $\delta_\Delta$ of a compact quantum group $(H, \Delta)$ on $A \circledast_\delta H$ is not defined diagonally, as $H$ may be noncommutative, but is instead the result of pointwise application of $\textrm{id} \otimes \Delta$. This apparent discrepancy from the topological case is remedied by the choice of boundary condition; for $f \in A \circledast_\delta H$, the requirement of $f(1) \in \delta(A)$ instead of $f(1) \in A \otimes \C$ is suited to this action. There is also a non-equivariant join of unital $C^*$-algebras,
\beu
A \circledast B = \{f \in C([0,1], A \otimes_\textrm{min} B): f(0) \in \C \otimes B, f(1) \in \C \otimes A\},
\eeu
and if $G$ is a compact group acting on $A$, the diagonal action $d$ on $A \circledast C(G)$ may be applied. Moreover, it is known from \cite{ba15} that $(A \circledast_\delta C(G), \delta_\Delta)$ and $(A \circledast C(G), d)$ are equivariantly isomorphic, and $(A \circledast C(\Z/2\Z), d)$ is also equivariantly isomorphic to $(\Sigma A, \alpha_\Sigma)$.

\begin{conj}[Baum-D\polhk{a}browski-Hajac, \cite{ba15} Conjecture 2.3 Type 1]\label{conj:BDH}
Let $A$ be a unital $C^*$-algebra with a free action of a nontrivial compact quantum group $(H, \Delta)$. Also, let $A \circledast_\delta H$ be the equivariant noncommutative join $C^*$-algebra of $A$ and $H$ with the induced free action of $(H, \Delta)$ given by $\delta_\Delta$. Then there is no $H$-equivariant $*$-homomorphism from $A$ to $A \circledast_\delta H$.
\end{conj} 

\begin{conj}[Baum-D\polhk{a}browski-Hajac, \cite{ba15} Conjecture 2.2]\label{conj:BDHtopcaseonly}
Let $X$ be a compact Hausdorff space with a continuous free action of a nontrivial compact Hausdorff group $G$. Then, for the diagonal action of $G$ on $X * G$, there does not exist a $G$-equivariant continuous map $f: X * G \to X$.
\end{conj}

It is known from Corollary 3.1 of \cite{vo05} that Conjecture \ref{conj:BDHtopcaseonly} holds for $G$ with a nontrivial torsion element, in which case it is sufficient to consider $G = \Z/p\Z$ for $p$ prime. In fact, \cite{vo05} predates \cite{ba15}. The proof relies on a stable index $\textrm{ind}_p(X) < \infty$, for which existence of an equivariant map $X \to Y$ implies $\textrm{ind}_p(X) \leq \textrm{ind}_p(Y)$, and $\textrm{ind}_p(X*\Z/p\Z) = \textrm{ind}_p(X) + 1$. This is also known to imply existence of a maximum $n$ such that $(\Z/p\Z)^{*n}$ maps equivariantly into $X$. We find that the topological result in \cite{vo05} passes to the $C^*$-algebraic setting in Corollary \ref{cor:mainkjoins}. That is, Conjecture \ref{conj:dabgen} holds, and the result generalizes to $\Z/k\Z$ actions for $k > 2$.

We place Corollary \ref{cor:mainkjoins} and its topological predecessor in the context of fixed point theorems on Stone-\v{C}ech compactifications, and during this pursuit we develop alternative proofs using an iteration strategy. A key consideration in these proofs is the recasting of freeness conditions in terms of spectral subspaces.

\begin{defn}[\cite{ph09}]\label{def:saturated}
Let $\alpha: G \to \textrm{Aut}(A)$ be a strongly continuous action of a compact Hausdorff abelian group $G$ on a unital $C^*$-algebra $A$. Then $\alpha$ is \textit{saturated} if for all $\tau$ in the Pontryagin dual $\widehat{G}$, the spectral subspace $A_\tau = \{a \in A: \alpha_g(a) = \tau(g) a \textrm{ for all } g \in G\}$ satisfies $A = \overline{A_\tau^* A A_\tau}$.
\end{defn}

A free $G$-action on $X$ induces a saturated action of $G$ on $A = C(X)$ (see \cite{ph09}). While the saturation condition differs from the Ellwood coaction freeness condition (\ref{eq:ellwood}), an easy computation shows that when $\Z/k\Z$ acts on $A$, saturation of the action is equivalent to freeness of the associated coaction. Note that the equivalence of freeness and saturation has been established in a much broader context (see \cite{ba13} and \cite{de13}). Moreover, examination of spectral subspaces shows that an action of $\Z/k\Z$ on $A$ produces a grading of $A$ by $\Z/k\Z$:
\beu
A = \bigoplus_{m \in \Z/k\Z} A_m, \hspace{1 in} A_m A_n \subset A_{m+n}.
\eeu
From Theorem 8.17 of \cite{mo93}, it follows that the associated coaction is free if and only if the grading is strong, in the sense that for all $m, n \in \Z/k\Z$,
\be\label{eq:stronggrading}
A_m A_n = A_{m+n}.
\ee
These various conditions related to freeness allow us to recover the unit $1 \in A_0$ from elements in nontrivial spectral subspaces $A_m$, a crucial technique in the alternative proofs of Conjecture \ref{conj:dabgen} (as part of Corollary \ref{cor:mainkjoins}). The alternative proofs provide a potential baseline for the last remaining case of Conjecture \ref{conj:BDHtopcaseonly}, as well as the associated classical subcase of Conjecture \ref{conj:BDH}.

 Finally, given an action $\beta: \Z/k\Z \to \textrm{Aut}(A)$, we introduce a type of twisted noncommutative join $A \circledast^\beta C^*(\Z/k\Z)$ that differs somewhat from the $C^*$-algebraic joins of \cite{da15b} in order to investigate a different flavor of noncommutative Borsuk-Ulam theory:
\begin{equation*} A \circledast^\beta C^*(\Z/k\Z) := \{f \in C([0, 1], A \rtimes_\beta \mathbb{Z}/k\mathbb{Z}): f(0) \in A, f(1) \in C^*(\Z/k\Z)\}. \end{equation*}

If $\alpha$ is another action of $\Z/k\Z$ on $A$ which commutes with $\beta$, then $A \circledast^\beta C^*(\Z/k\Z)$ may be equipped with an associated action $\widetilde{\alpha}$, motivated by (\ref{eq:newaction}). In fact, if $\beta$ is trivial, $A \circledast^\beta C^*(\Z/2\Z)$ is actually isomorphic to $\Sigma A$, and these two actions are the same. If $\beta$ is trivial and $k \in \Z^+$ is arbitrary, then $A \circledast^\beta C^*(\Z/k\Z)$ coincides with the (non-equivariant) $C^*$-algebraic join $A \circledast C(\Z/k\Z)$, equipped with the diagonal action of $\Z/k\Z$.

When $A \circledast^\beta C^*(\Z/k\Z)$ is formed from a nontrivial action $\beta$, topological reductions and iteration techniques are not sufficient to prove an analogous Borsuk-Ulam theorem, and additional assumptions on $\beta$ are necessary. However, $\theta$-deformed spheres may be used to construct natural examples where $\beta$ is a nontrivial action and no equivariant maps from $A$ to $A \circledast^\beta C^*(\Z/k\Z)$ exist. Whether this phenomenon is unique to restricted families of examples is still unclear.

\section{Classical Joins}

There are many topological extensions of the Borsuk-Ulam theorem, with varying degrees of generality; we consider here a generalization given by A. Dold in 1983. Here $\textrm{dim}(\cdot)$ and $\textrm{conn}(\cdot)$ refer to covering dimension and connectivity, respectively. 

\begin{thm}[Dold, \cite{do83} Main Theorem]\label{thm:dabtopfinite}
Suppose $X$ and $Y$ are paracompact Hausdorff spaces with free actions of a finite group $G$.
\begin{enumerate}
\item If there is a continuous, equivariant map $f: X \to Y$, then $\textrm{dim}(Y) \geq 1 + \textrm{conn}(X)$.
\item\label{item:selfmaps} If $\textrm{dim}(X) < \infty$, then every continuous, equivariant map $g: X \to X$ is homotopically nontrivial.
\end{enumerate}
\end{thm}

Similar results and comments may be found in \cite{ca77}, Theorem 4 for CW-complexes, \cite{go82} for manifolds, and \cite{ya13}, Remark 17 for the noncommutative ($KK$-theoretic) point of view. When $X$ is assumed compact, the dimension restriction may be removed.

\begin{thm}[Volovikov, \cite{vo05} Corollary 3.1]\label{thm:volovikov}
Suppose $X$ is a compact (or paracompact and finite-dimensional, as above) Hausdorff space, $G$ has nontrivial torsion elements, and $G$ acts freely on $X$. Then there is no continuous, equivariant map from $X*G$ to $X$. That is, every continuous, equivariant map $f: X \to X$ is homotopically nontrivial.
\end{thm}

If $G = \Z/p\Z$ for $p$ prime, then Theorem \ref{thm:volovikov} and condition (\ref{item:selfmaps}) of Theorem \ref{thm:dabtopfinite} have  converses in the following sense. Suppose that an action of $\Z/p\Z$ on $X$ is not free, so that there is a point $x_0 \in X$ which is fixed by every group element. Then the constant map sending every point in $X$ to $x_0$ is both equivariant and homotopically trivial by default.

The proof in \cite{vo05} (see also \cite{mu16}) of Theorem \ref{thm:volovikov} relies on a stable index $\textrm{ind}_p(X) < \infty$ for free actions of $\Z/p\Z$, $p$ prime. This means that $\textrm{ind}_p(X*\Z/p\Z) = \textrm{ind}_p(X) + 1$ and $\textrm{ind}_p(X) \leq \textrm{ind}_p(Y)$ whenever there is an equivariant map $X \to Y$. Stability of the index is known to not only prevent the existence of equivariant maps $X * \Z/p\Z \to X$, but also to bound the possible $n$ for which $(\Z/p\Z)^{*n}$ maps equivariantly to $X$. Our proofs of Conjecture \ref{conj:dabgen} and its $\Z/k\Z$ generalization move this result from the compact Hausdorff setting to the $C^*$-algebraic setting. Note that the iterated joins of $\Z/2\Z$ are just spheres $\S^k$, and the induced diagonal action on the join is the antipodal action on $\S^k$.

\begin{thm}\label{thm:mainkjoinsHOM}
Let $A$ be a unital $C^*$-algebra with a free action $\alpha$ of $\Z/k\Z$ and suppose $\phi = \phi_1: A \to A$ is an $\alpha$-equivariant, unital $*$-homomorphism. Then there does not exist a path $\phi_t$ of unital $*$-homomorphisms on $A$, continuous in the pointwise norm topology, connecting $\phi_1$ to a one-dimensional representation $\phi_0: A \to \C$. 
\end{thm}
\begin{proof}
Suppose $\phi_t: A \to A$ is a path of unital $*$-homomorphisms such that $\phi_0$ is $\C$-valued and $\phi_1$ is $\alpha$-equivariant. Since $\phi_0$ is unital and therefore is not the zero map, it follows that the closed ideal $I \subset A$ generated by commutators $ab - ba$ is proper. That is, $B := A/I$ is the algebra of continuous, complex-valued  functions on a (nonempty) compact Hausdorff space $X$. Denote the equivalence class of $a \in A$ in $B$ by $[a]$.

For simplicity of notation, identify $\alpha$ with the automorphism associated to the standard generator of $\Z/k\Z$. It follows that any element of $I$ is annihilated in the composition
\begin{equation*}
A \xrightarrow{\alpha} A \to B,
\end{equation*}
since $B$ is commutative. Therefore, $\alpha$ descends to a unital $*$-homomorphism
\begin{equation*}
\beta: B \to B
\end{equation*}
satisfying $\beta([a]) = [\alpha(a)]$. Consequently, the order of $\beta$ must also divide $k$, and $\beta$ defines a $\Z/k\Z$ action on $A/I$. Also, since $\alpha$ is free, it follows that $\alpha$ is saturated, so for any $k$th root of unity $\omega$, the spectral subspace
\begin{equation*}
A_\omega = \{ a \in A: \alpha(a) = \omega a\}
\end{equation*}
satisfies $A = \overline{A_\omega^* A A_\omega}$. Since $a \in A_\omega$ implies $[a] \in B_\omega$, it immediately follows that
\begin{equation*}
B = \overline{B_\omega^* B B_\omega},
\end{equation*}
and $\beta$ is saturated. That is, the associated action on the compact Hausdorff space $X$ is free. 

We may apply the same quotient by the commutator ideal $I$, in both the codomain and domain, to the path $\phi_t: A \to A$. This produces a family of unital $*$-homomorphisms $\psi_t: B \to B$ meeting similar boundary conditions: $\psi_0$ is $\C$-valued and $\psi_1$ is $\beta$-equivariant. Note that the homomorphisms $\psi_t$ remain continuous in the pointwise norm topology. Finally, the maps $\psi_t$ are dual to a homotopy of maps on $X$, connecting an equivariant map to a constant map. This contradicts Theorem \ref{thm:volovikov}.
\end{proof}

In contrast to the topological discussion after Theorem \ref{thm:volovikov}, Theorem \ref{thm:mainkjoinsHOM} does not have a converse, even if $k$ is prime. That is, if an action of $\Z/k\Z$ on a unital $C^*$-algebra $A$ is not free, we cannot claim that there exists some equivariant map from $A$ to $A$ which may be connected to a one-dimensional representation. For example, take $A$ to be noncommutative and simple, and equip $A$ with the trivial action, which is certainly not free. In this case, the converse fails because one-dimensional representations of $A$ simply do not exist.

\begin{cor}\label{cor:mainkjoins}
If $A$ is a unital $C^*$-algebra with a free action $\alpha$ of $\Z/k\Z$, then there are no equivariant, unital $*$-homomorphisms from $A$ to $A \circledast C(\Z/k\Z)$ with the diagonal action. Similarly, if $k = 2$, there are no $(\alpha, \alpha_\Sigma)$-equivariant, unital $*$-homomorphisms from $A$ to $\Sigma A$, so Conjecture \ref{conj:dabgen} holds.
\end{cor}
\begin{proof}
Suppose an equivariant map $A \to A \circledast C(\Z/k\Z)$ exists. Pointwise evaluation at $0 \in \Z/k\Z$ on the $C(\Z/k\Z)$ tensorand produces a path of unital $*$-homomorphisms on $A$ connecting an $\alpha$-equivariant map to a one-dimensional representation, which contradicts Theorem \ref{thm:mainkjoinsHOM}. For $k=2$, note that $\Sigma A$ and $A \circledast C(\Z/2\Z)$ are equivariantly isomorphic.
\end{proof}

Recall that when $\Z/p\Z$, $p$ prime, acts on a compact Hausdorff space $X$, the index in \cite{vo05} bounds the set of $n$ for which there is an equivariant map from $(\Z/p\Z)^{*n}$ to $X$. Here we present an iteration procedure which simultaneously produces alternative proofs of Corollary \ref{cor:mainkjoins} along these lines and sets up an analysis of the topological case, Theorem \ref{thm:volovikov}, in terms of fixed point problems. This suggests a possible direction for the remaining case of Conjecture \ref{conj:BDHtopcaseonly} and remaining classical subcase of Conjecture \ref{conj:BDH}.

\begin{lem}\label{lem:iteration}
Let $G$ be a compact Hausdorff group acting on a unital $C^*$-algebra $A$, and suppose that there is an equivariant, unital $*$-homomorphism $\psi: A \to A \circledast C(G)$, where the join is equipped with the diagonal action. Then for any $n \in \Z^+$, there is an equivariant, unital $*$-homomorphism $\psi_n: A \to C(G^{*n})$.
\end{lem}
\begin{proof}
If $\psi: A \to A \circledast C(G)$ is equivariant for a $G$-action on $A$ and the associated diagonal action on $A \circledast C(G)$, then we may join $\psi$ with the identity map on $C(G)$. This produces a map $A \circledast C(G) \to A \circledast C(G) \circledast C(G)$ which is again equivariant for the diagonal actions. Repeating this procedure, we may use as many joins as we wish and compose the maps in a chain
\begin{equation*}
A \to A \circledast C(G) \to A \circledast C(G) \circledast C(G) \to  A \circledast C(G) \circledast C(G) \circledast C(G) \to\ldots
\end{equation*}
that continues indefinitely. The $C^*$-join is dual to the topological join, so compositions produce equivariant maps
\begin{equation*}
A \to A \circledast C(G^{*n}),
\end{equation*}
and evaluating at the $t = 0$ boundary of the join produces equivariant maps 
\begin{equation*}
\psi_{n}: A \to C(G^{*n})
\end{equation*}
for any $n$, as desired. 
\end{proof}

The iteration procedure in Lemma \ref{lem:iteration} may be used to produce an alternative proof of Corollary \ref{cor:mainkjoins}.

\vspace{.1 in}

\noindent \textbf{Alternative Proof A of Corollary \ref{cor:mainkjoins}.} Suppose for the sake of contradiction that $G = \Z/k\Z$ acts freely on a unital $C^*$-algebra $A$, and that $\phi: A \to A \circledast C(G)$ is an equivariant, unital $*$-homomorphism. By Lemma \ref{lem:iteration}, there is an equivariant map $\psi_n: A \to C(G^{*n})$. Since the action on $A$ is free, the associated grading $A = \bigoplus\limits_{\tau \in \widehat{G}} A_\tau \cong \bigoplus\limits_{m \in G} A_m$ by spectral subspaces is strong in the sense of (\ref{eq:stronggrading}), so we may conclude that $1 \in A_0 = A_m A_{-m}$ for any $m \in G$. Fix $m$ which generates $G$ and select $a_1, \ldots, a_d \in A_m$, $b_1, \ldots, b_d \in A_{-m}$ with $1 = \sum\limits_{i=1}^d a_i b_i$. It follows that 
\be\label{eq:strongmet}
1 = \sum\limits_{i=1}^d \psi_{2d+1}(a_i)\psi_{2d+1}(b_i)
\ee
holds in $C(G^{*2d+1})$. Equip $G^{*2d+1}$ with the diagonal action of $G$, and equip $\S^{2d-1}$ with rotation by a primitive $k$th root of unity, so that by (\ref{eq:strongmet}), there is a continuous, equivariant map from $G^{*2d+1}$ to $\S^{2d-1}$. Both actions are free, $G^{*2d+1}$ is $(2d-1)$-connected by Proposition 4.4.3 of \cite{ma03}, and $\textrm{dim}(\S^{2d-1}) = 2d-1$, so we have reached a contradiction of Theorem \ref{thm:dabtopfinite}. (Alternatively, a contradiction can be reached in a similar way using the index.) \qed

\vspace{.1 in}

Note that when $G = \Z/2\Z$, the iterated join $G^{*2d+1}$ is the sphere $\S^{2d}$, so nonexistence of equivariant maps $G^{*2d+1} \to \S^{2d-1}$ in this case is just the Borsuk-Ulam theorem in even dimension. Similarly, the Borsuk-Ulam theorem implies that if the antipodal action on spheres is extended to the Stone-\v{C}ech compactification $\beta \S^{\infty}$ of the infinite sphere, then the resulting action has a fixed point. More generally, it is known from \cite{fe02} that for prime $p$, the diagonal action of $\Z/p\Z$ on the compactified infinite join $\beta (\Z/p\Z)^{*\infty} = \beta \left(\lim\limits_{\rightarrow} (\Z/p\Z)^{*n}\right )$ is not free. This allows us to produce a slightly different alternative proof of Corollary \ref{cor:mainkjoins}.

\vspace{.1 in}

\noindent \textbf{Alternative Proof B of Corollary \ref{cor:mainkjoins}.} Suppose $G$ is a finite cyclic group acting on a unital $C^*$-algebra $A$, and that an equivariant map $\phi: A \to A \circledast C(G)$ exists. By reducing to a subgroup of $G$, we may suppose that $G = \Z/p\Z$ for $p$ prime. An unmentioned feature of the iteration
\begin{equation*}
A \to A \circledast C(G) \to A \circledast C(G) \circledast C(G) \to A \circledast C(G) \circledast C(G) \circledast C(G) \to \ldots
\end{equation*}
used in Lemma \ref{lem:iteration} is that the resulting maps $\psi_n: A \to C(G^{*n})$ are consistent as $n$ grows. That is, they commute with the quotients $C(G^{*n+1}) \to C(G^{*n})$ dual to the embeddings $G^{*n} \to G * G^{*n} = G^{*n+1}$. We may replace the domain $A$ with its largest (nontrivial) commutative quotient $C(X)$, producing a collection of continuous, equivariant maps
\beu
f_n: G^{*n} \to X
\eeu
which commute with the embeddings $G^{*n} \to G^{*n+1}$. These extend to the direct limit $G^{*\infty}$ as a continuous, equivariant map $f: G^{*\infty} \to X$. Since $X$ is compact, we may extend the map $f$ to an equivariant map $\beta f: \beta G^{*\infty} \to X$. The domain has a fixed point by \cite{fe02}, and pushing a fixed point through $\beta f$ shows that the action of $G$ on $X$ is not free in the topological sense. Consequently, the associated action of $G$ on the $C^*$-algebra $C(X)$ is not saturated, and neither is the action on $A$, as the image of a spectral subspace of $A$ under the quotient $A \to C(X)$ is contained in the same spectral subspace of $C(X)$. Finally, we conclude that the action on $A$ does not meet the coaction freeness condition, completing a proof by contraposition. \qed

\vspace{.1 in}

The results in \cite{fe02} extend Milnor's study of the infinite join $G^{*\infty}$ as a universal $G$-bundle in \cite{mi56a} and \cite{mi56b}. Moreover, fixed points of induced maps on Stone-\v{C}ech compactifications have been well-studied (\cite{bo07}, \cite{po00}, \cite{vado93}). While the fixed point results on $\beta G^{*\infty}$ from \cite{fe02} used in Alternative Proof B do follow from the connectivity estimates on finite joins $G^{*n}$, which are actually used directly in Alternative Proof A, it is important to note that the methods are functionally different. When considering $G$ which are no longer finite, but only compact, it is not known if estimates on finite joins are sufficient, or if one needs to use the much stronger tool in Alternative Proof B (or something more). 

For the remaining case of Conjecture \ref{conj:BDHtopcaseonly}, or of Conjecture \ref{conj:BDH} in the classical setting $H = C(G)$, it is certainly possible to reduce the problem to subgroups of $G$. Therefore, it suffices to consider $G$ which are abelian and torsion-free. It is not yet clear if the methods of Alternative Proof A or Alternative Proof B of Corollary \ref{cor:mainkjoins} generalize to this context. We describe here the possible generalization of Alternative Proof A, which concerns only finitely many joins at a time.

\begin{ques}\label{ques:countingisfun}
Let $G$ be a compact, Hausdorff, abelian, torsion-free, nontrivial group. Let $G$ act on itself by multiplication and extend this action diagonally to iterated joins. Any such action is free, so the associated action on $A = C(G^{*n})$ is free. Consequently, for any $\tau$ in the Pontryagin dual $\widehat{G}$, the closed ideal generated by the spectral subspace $A_\tau$ is $A$. Does the number of elements in the $\tau$ spectral subspace of $C(G^{*n})$ required to approximate the multiplicative identity $1$ remain bounded as $n$ increases? 
\end{ques}

If Question \ref{ques:countingisfun} can generally be answered in the negative, then this fact would say that the behavior gleaned from the index $\mathrm{ind}_p$ in \cite{vo05} for $\Z/p\Z$ actions carries over to a more general setting, and that the spiritual successor of this technique, Alternative Proof A, would extend to such $G$. On the other hand, if the answer is not always negative, then one might benefit from an examination of $G^{*\infty}$ using more sophisticated tools than eigenfunction counting on finite joins.

Now, compact, Hausdorff, abelian, torsion-free groups are classified in \cite{ho13} and are fairly exotic, such as $G = \widehat{(\Q, \mathcal{D})}$, where $\mathcal{D}$ is the discrete topology. Since this group is comprised of characters with domain $\Q$, any $q \in \Q$ produces a character on $G$, namely, the evaluation $\mathrm{ev}_q: f \mapsto f(q)$. By Pontryagin duality, all characters of $G$ are of this form, and it would be of interest to settle Question \ref{ques:countingisfun} in this special case.

\section{A Noncommutative Suspension and $\Z/k\Z$-Join}

Certain $\theta$-deformed noncommutative spheres (\cite{ma91}, \cite{na97}) may be written as unreduced suspensions based on an examination of their presentations. Here $\rho \in M_n(\C)$ denotes a matrix which may be written entrywise as $\rho_{jk} = e^{2 \pi i \theta_{jk}}$ for some antisymmetric $\theta \in M_n(\R)$.
\beu
C(\S^{2n-1}_\rho) = C^*\left(z_1, \ldots, z_n \hspace{2 pt} \left| \hspace{2 pt} z_j z_j^* = z_j^* z_j, z_k z_j = \rho_{jk}z_j z_k, \sum_{j=1}^n z_jz_j^* = 1 \right.\right)
\eeu
\beu \ba
C(\S^{2n}_\rho) &= C^*\left(z_1, \ldots, z_n, x \hspace{2 pt} \left| \hspace{2 pt} z_j z_j^* = z_j^* z_j, x = x^*, z_k z_j = \rho_{jk}z_j z_k, x z_j = z_j x, x^2 + \sum_{j=1}^n z_jz_j^* = 1 \right.\right) \\
& \cong C(\S^{2n+1}_\omega)/\langle z_{n+1} - z_{n+1}^* \rangle, \hspace{.15 in} \omega_{jk} = \left\{\begin{array}{cccc} \rho_{jk}, & j \leq n \textrm{ and } k \leq n \hfill \\ 1, & j = n+1 \textrm{ or } k = n+1 \end{array}\right.
\ea \eeu
For example, every even $\theta$-sphere is realized as the unreduced suspension $C(\S^{2n}_\rho) \cong \Sigma C(\S^{2n - 1}_\rho)$, and if $C(\S^{2n + 1}_\omega)$ is such that $z_{n + 1}$ is central, then $C(\S^{2n + 1}_\omega) \cong \Sigma C(\S^{2n}_\rho)$ for an appropriately chosen $\rho$. In these cases, the antipodal action, which by definition negates each generator of the presentation, suspends to the antipodal action. However, this procedure does not capture all $(2n + 1)$-spheres, and it seems reasonable that we might reach some odd sphere $C(\S^{2n + 1}_\omega)$ with anticommutation relations on $z_{n + 1}$ by passing through a different type of suspension procedure twice, starting from $C(\S^{2n-1}_\rho)$. In particular, at each suspension we wish to add a self-adjoint generator that is not central, and in doing so we also develop a noncommutative join $A \circledast^\alpha C^*(\Z/k\Z)$ that is different from the $C^*$-algebraic joins in \cite{da15b}.

We adopt the convention that the group $C^*$-algebra $C^*(\Z/k\Z)$ and crossed products $A \rtimes_\alpha \Z/k\Z$ will have group elements denoted by $\delta^n$, where $\delta$ is a prescribed generator of $\Z/k\Z$. For example, a general element of $A \rtimes_\alpha \Z/k\Z$ is $a_0 + a_1 \delta + a_2 \delta^2 + \ldots + a_{k - 1} \delta^{k - 1}$. Also, an action $\alpha$ of $\Z/k\Z$ on $A$ is identified with its generating homomorphism $\alpha: A \to A$ of order dividing $k$.

\begin{lem}\label{lem:crosstorus}
Let $A_\theta$ denote a quantum 2-torus whose unitary generators anticommute. If $\alpha$ denotes the antipodal action on $C(\S^1)$ and $\widehat{\alpha}$ denotes the dual action on the crossed product $C(\S^1) \rtimes_\alpha \Z/2\Z$ (so $\widehat{\alpha}(f + g\delta) = f - g\delta$), then $A_\theta \cong \{h \in C(\S^1, C(\S^1) \rtimes_\alpha \Z/2\Z): \widehat{\alpha}(h(-w)) = h(w) \textrm{ for all } w \in \S^1 \}.$
\end{lem}
\begin{proof}
Denote the standard generator of $C(\S^1)$ by $\chi$, and let $\gamma: A_\theta \to C(\S^1, C(\S^1) \rtimes_\alpha \Z / 2\Z)$ be defined by $\gamma(U_j) = g_j$, where $g_1(w) = \chi$ is a constant function in $w$ and $g_2(w) = w \delta$. Note that in both images, the symmetry condition $\widehat{\alpha}(g_j(-w)) = g_j(w)$ imposed on the target space is satisfied. Now, $\gamma$ is also trace-preserving, where $A_\theta$ has the usual trace and $C(\mathbb{S}^1, C(\mathbb{S}^1) \rtimes_\alpha \Z/2\Z)$ is given the trace $\tau(f + g \delta) = \int \int f(w)[z] d\mu(z) d\mu(w)$, where $\mu$ is normalized Lebesgue measure on $\S^1$. The usual trace on $A_\theta$ is faithful, so $\gamma$ is injective.

To show that $\textrm{Ran}(\gamma)$ includes all $h \in C(\S^1, C(\S^1) \rtimes_\alpha \Z / 2\Z)$ with $\widehat{\alpha}(h(-w)) = h(w)$, it is sufficient to prove that $\textrm{Ran}(\gamma)$ includes a dense subset of the target, as $C^*$-homomorphisms have closed range. Target elements are of the form $w \mapsto f(w) + g(w) \delta$ where $f(w), g(w) \in C(\S^1)$ have $f(-w)[z] = f(w)[z]$ and $g(-w)[z] = -g(w)[z]$. Define $F, G \in C(\T^2)$ by $F(w, z) = f(w)[z]$ and $G(w, z) = g(w)[z]$, so that $F(-w, z) = F(w, z)$ and $G(-w, z) = -G(w, z)$. Both $F$ and $G$ are uniformly approximable by $*$-polynomials $P(w, z) = p(w)[z]$, $Q(w, z) = q(w)[z]$ that meet the same symmetry conditions in the first coordinate, and $w \mapsto p(w) + q(w) \delta$ is certainly in the range of $\gamma$.
\end{proof}

\begin{lem}\label{lem:crosssphere}
If $C(\S^3_\rho)$ denotes the $\theta$-deformed $3$-sphere with $z_2 z_1 = -z_1 z_2$, and $\alpha$ denotes the antipodal action on $C(\S^1)$, then $C(\S^3_\rho) \cong \{f \in C(\overline{\D}, \C(\S^1) \rtimes_\alpha \Z/2\Z): \widehat{\alpha}(f(-w)) = f(w) \textrm{ for all } w,$ $\textrm {and if } w \in \del \D, \textrm{then } f(w) \in \C + \C \delta\}$.
\end{lem}
\begin{proof}
Similar to the argument of Lemma \ref{lem:crosstorus}, let $\phi: C(\S^3_\rho) \to C(\overline{\mathbb{D}}, C(\S^1) \rtimes_\alpha \Z/2\Z)$ be defined by $\phi(z_1)[w] = \sqrt{1 - |w|^2} z$ and $\phi(z_2)[w] = w \delta$. By viewing $\overline{\D}$ as a union of concentric circles shrinking to a point, any element of $C(\overline{\mathbb{D}}, C(\S^1) \rtimes_\alpha \Z/2\Z)$ meeting the symmetry and boundary conditions of the target is a continuous family of functions $\psi_t: \S^1 \to C(\S^1) \rtimes_\alpha \Z/2\Z$ indexed by $ t \in [0, 1]$, where $\psi_0$ outputs a single member of $C(\S^1)$, $\psi_1$ is $\C + \C \delta$-valued, and $\widehat{\alpha}(\psi_t(-w)) = \psi_t(w)$. Lemma \ref{lem:crosstorus} shows that each $\psi_t$ represents an element of the quantum torus $A_\theta$ with two anticommuting generators $U_1$ and $U_2$, where $\psi_0$ gives an element of $C^*(U_1)$ and $\psi_1$ gives an element of $C^*(U_2)$. Now, $(\psi_t)_{t \in [0, 1]}$ produces a single element of $C(\mathbb{S}^3_\rho)$ by \cite{na97}, Theorem 2.5, and a straightforward computation shows that this procedure has produced an inverse for $\phi$.
\end{proof}

If the anticommuting sphere $C(\S^3_\rho)$ is viewed as an algebra of functions on $\overline{\D}$ as in Lemma \ref{lem:crosssphere}, then $z_2 - z_2^*$ vanishes on $[-1, 1] \subset \overline{\D}$, and the quotient $C(\S^3_\rho) / \langle z_2 - z_2^* \rangle$ is isomorphic to an algebra of functions on $[-1, 1]$:
\begin{equation*} \{f \in C([-1, 1], C(\S^1) \rtimes_\alpha \Z/2\Z): \widehat{\alpha}(f(-t)) = f(t) \textrm{ for all } t \in [-1, 1], f(\pm1) \in \C + \C \delta\}. \end{equation*}
The symmetry condition renders the domain $[-1, 1]$ redundant, so
\begin{equation*} C(\S^3_\rho) / \langle z_2 - z_2^* \rangle \cong \{f \in C([0, 1], C(\S^1) \rtimes_\alpha \Z/2\Z): f(0) \in C(\S^1), f(1) \in \C + \C \delta \}. \end{equation*}

In other words, to add a self-adjoint, anticommuting generator to $C(\S^1)$, it suffices to look at a space of functions into a crossed product. There is no obstruction to extending this technique to higher-dimensional odd $\theta$-deformed spheres $C(\S^{2n-1}_\rho)$. Moreover, the antipodal action $\alpha$ could be replaced by another action which only negates certain generators $z_i$ of $C(\S^{2n-1}_\rho)$; the new self-adjoint generator will then commute or anticommute with $z_i$ depending on the action. Written in full generality, this procedure may be considered a type of noncommutative unreduced suspension.

\begin{defn}
Let $A$ be a unital $C^*$-algebra with an action $\beta$ of $\Z/2\Z$. The $\beta$\textit{-twisted noncommutative unreduced suspension} of $A$ is defined as 
\begin{equation*}
\Sigma^\beta A := \{f \in C([0,1], A \rtimes_\beta \Z/2\Z): f(0) \in A, f(1) \in \C + \C\delta = C^*(\Z/2\Z)\}.
\end{equation*}
\end{defn}

This twisted unreduced suspension does not appear to match up with the double unreduced suspension found in \cite{ho08}. When $\beta$ is the trivial action, $\Sigma^\textrm{triv} A$ is isomorphic to $\Sigma A$, where the crossed product encodes homogeneity classes of functions. Specifically,
$\Sigma A \cong \Sigma^{\textrm{triv}} A$ via $f(t) \mapsto \frac{f(t) + f(-t)}{2} + \frac{f(t) - f(-t)}{2}\delta$, where we note that while $f(t)$ is defined on $[-1, 1]$, the image functions are only defined on $[0, 1]$. The inverse image of $h + k \delta \in \Sigma^\textrm{triv}A$ is then $H + K$, where $H$ is the even extension of $h$ and $K$ is the odd extension of $k$.

For the usual unreduced suspension, any $\Z/2\Z$ action on $A$ could be extended to $\Sigma A$ in a natural way, by applying the original action pointwise and composing with $t \mapsto -t$. For $\Sigma^\beta A$, we may extend an action $\alpha$ only if it commutes with $\beta$, and in this case we apply $\alpha$ pointwise on $A \rtimes_\beta \Z/2\Z$ and compose with $\widehat{\beta}$, which negates the group element $\delta$.
\beu
\widetilde{\alpha}: f(t) + g(t) \delta \mapsto \alpha(f(t)) - \alpha(f(t))\delta
\eeu

These ideas may be extended easily from suspensions to joins. If $X$ is a compact Hausdorff space on which $G$ acts, then the (non-equivariant) noncommutative join of \cite{da15b} is motivated by the identification
\beu \ba
C(X * G) 	&\cong \{f \in C([0, 1], C(X \times G)): f(0) \textrm{ is } G\textrm{-independent}, f(1) \textrm{ is } X\textrm{-independent} \} \\
		&\cong \{f \in C([0, 1], C(X) \otimes C(G)): f(0) \in C(X), f(1) \in C(G) \}. \\
\ea \eeu
A different noncommutative join for finite cyclic groups may be developed by realizing $C(X) \otimes C(\Z/k\Z)$  as a trivial crossed product
\beu
C(X * \Z/k\Z) \cong \{f \in C([0, 1], C(X) \rtimes_\textrm{triv} \Z/k\Z): f(0) \in C(X), f(1) \in C^*(\Z/k\Z) \}.
\eeu

\begin{defn}\label{defn:twistedjoin}
Let $A$ be a unital $C^*$-algebra with a $\Z/k\Z$ action $\beta$. Then the $\beta$\textit{-twisted noncommutative join of} $A$ \textrm{ and } $C^*(\Z/k\Z)$ is defined as 
\begin{equation*}
A \circledast^\beta C^*(\Z/k\Z) := \{f \in C([0, 1], A \rtimes_\beta \Z/k\Z): f(0) \in A, f(1) \in C^*(\Z/k\Z) \}.
\end{equation*}
\end{defn}

Note that $A \circledast^\textrm{triv} C^*(\Z / k \Z)$ is isomorphic to the (non-equivariant) noncommutative join $A \circledast C(\Z / k\Z)$ of \cite{da15b}. An action $\alpha$ of $\Z/k\Z$ on $A$ which commutes with $\beta$ may be extended to $A \circledast^\beta C^*(\Z/k\Z)$ by applying $\alpha$ pointwise and also applying the dual action of $\beta$. Since the Pontryagin dual of $\Z/k\Z$ is isomorphic to $\Z/k\Z$ in a non-canonical way, we fix a primitive $k$th root of unity $\omega = e^{2 \pi i / k}$. An element of $A \circledast^\beta C^*(\Z/k\Z)$ may be viewed as $f_0 + f_1 \delta + \ldots + f_{k - 1} \delta^{k - 1}$ where $f_0, f_1, \ldots, f_n$ are functions from $[0, 1]$ into $A$. Then we extend $\alpha$ to $A \circledast^\beta C^*(\Z/k\Z)$ as $\widetilde{\alpha}$, just as in the $\Z/2\Z$ case:
\beu 
\widetilde{\alpha}(f_0 + f_1\delta + \ldots + f_{k - 1}\delta^{k - 1})(t) := \alpha(f_0(t)) + \omega \alpha(f_1(t)) \delta + \ldots + \omega^{k - 1} \alpha(f_{k - 1}(t)) \delta^{k - 1}.
\eeu

The isomorphism $C(\Z/k\Z) \cong C^*(\Z/k\Z)$ may be written in such a way that the dual action of $\widehat{\Z/k\Z} \cong \Z/k\Z$ on $C^*(\Z/k\Z)$ is equivalent to the translation action of $\Z/k\Z$ on $C(\Z/k\Z)$. We ask the following question in analogy with \cite{ba15} and \cite{da15}.

\begin{ques}\label{ques:NCjoin}
Suppose $A$ is a unital $C^*$-algebra with two commuting $\Z/k\Z$ actions $\alpha$ and $\beta$. If $\alpha$ is saturated and $\beta$ is unsaturated, must there not exist any unital, $(\alpha, \tilde{\alpha})$-equivariant $*$-homomorphisms from $A$ to $A \circledast^\beta C^*(\Z/k\Z)$?
\end{ques}

In other words, we ask if the noncommutative Borsuk-Ulam theorem in Corollary \ref{cor:mainkjoins} can survive when the join is twisted by a crossed product. While the construction in Definition \ref{defn:twistedjoin} can certainly be extended to more general Pontryagin duals $C(G) \cong C^*(\Gamma)$, with tensor products $A \otimes C^*(\Gamma)$ in the classical join replaced by crossed products $A \rtimes \Gamma$, we limit ourselves to the finite cyclic case as a base point, as this is the current state of the non-twisted Borsuk-Ulam results in $C^*$-algebraic language. In the twisted setting, the $\theta$-deformed spheres may be used as examples. Any odd $\theta$-deformed sphere $C(\S^{2n - 1}_\rho)$ admits $\Z/k\Z$ actions which rotate the generators $z_j$ by roots of unity. Fix such an action $\beta$ which maps $z_j$ to $\omega_j z_j$, where each $\omega_j$ is a not necessarily primitive $k$th root of unity. If $\alpha$ denotes another such action, using only primitive $k$th roots of unity, then $\alpha$ is a saturated action, $\alpha$ and $\beta$ commute, and $\alpha$ extends to an action $\widetilde{\alpha}$ on $C(\S^{2n-1}_\rho) \circledast^\beta C^*(\Z/k\Z)$.

\begin{exam}
Let $C(\S^{2n - 1}_\rho)$ be equipped with coordinate rotations $\alpha$ and $\beta$, as above, where $\alpha$ is saturated. Then there is no unital, $(\alpha, \widetilde{\alpha})$-equivariant $*$-homomorphism from $C(\S^{2n - 1}_\rho)$ to the join $C(\S^{2n - 1}_\rho) \circledast^\beta C^*(\Z/k\Z)$.
\end{exam}
\begin{proof}
Suppose $\psi: C(\S^{2n-1}_\rho) \to C(\S^{2n-1}_\rho) \circledast^\beta C^*(\Z/k\Z)$ is an $(\alpha, \widetilde{\alpha})$-equivariant, unital $*$-homomorphism. For each $t \in [0,1]$, $\mathrm{ev}_t \circ \psi: C(\S^{2n-1}_\rho) \to C(\S^{2n-1}_\rho) \rtimes_\beta \Z/k\Z$ is a $K_1$-trivial map, as for any $U \in U_k(C(\S^{2n-1}_\rho))$, $\mathrm{ev}_t(\psi(U))$ is homotopy equivalent inside $U_k(C(\S^{2n-1}_\rho) \rtimes_\beta \Z/k\Z)$ to $\mathrm{ev}_1(\psi(U)) \in U_k(C^*(\Z/k\Z))$, and $K_1(C^*(\Z/k\Z))$ is the trivial group. However, $\iota_*: K_1(C(\S^{2n-1}_\rho)) \to K_1(C(\S^{2n-1}_\rho) \rtimes_\beta \Z/k\Z)$ is injective, where $\iota: C(\S^{2n-1}_\rho) \to C(\S^{2n-1}_\rho) \rtimes_\beta \Z / k \Z$ is the inclusion map. It follows that the $(\alpha, \alpha)$-equivariant map $\mathrm{ev}_0 \circ \psi: C(\S^{2n-1}_\rho) \to C(\S^{2n-1}_\rho)$ is $K_1$-trivial, which contradicts \cite{pa15a}, Corollary 4.12.
\end{proof}

In this case, the $K$-theory invariant which produces Borsuk-Ulam theorems on $C(\S^{2n-1}_\rho)$ retains its potency in the crossed product. The action $\beta$ considered above can be either saturated or unsaturated, but the generating homomorphism is always in the same path class within $\textrm{Aut}(C(\S^{2n-1}_\rho))$ as the identity automorphism. The iteration technique in the previous section is not applicable to Question \ref{ques:NCjoin}, and we require $\beta$ to be unsaturated to rule out some simple counterexamples. If $\alpha = \beta$ were allowed to be saturated, one could consider iterated joins of $C^*(\Z/2\Z)$ in analogy with the usual unreduced suspension, but this method will surely fail due to the following example.

\begin{exam}
Let $B_n$ be the universal, unital $C^*$-algebra generated by self-adjoint elements $x_1, \ldots, x_{n + 1}$ which pairwise anticommute and satisfy $x_1^2 + \ldots + x_{n + 1}^2 = 1$. Equip $B_n$ with a $\Z/2\Z$ action $\alpha$ that negates each generator. Then there is an odd, self-adjoint, unitary element $x_1 + \ldots + x_{n + 1} \in B_n$, and therefore there is also an equivariant homomorphism $\psi_n: B_0 = C(\S^0) \to B_n$ defined by $\psi_n(x_1) = x_1 + \ldots + x_{n+1}$. Further, let $\phi: B_1 \to B_0 \circledast^\alpha C^*(\Z / 2\Z)$ be defined by $\phi(x_1)[t] = \sqrt{1-t^2}x_1$ and $\phi(x_2)[t] = t\delta$. Then $\phi \circ \psi_1: B_0 \to B_0 \circledast^\alpha C^*(\Z / 2\Z) = \Sigma^\alpha B_0$ is $(\alpha, \widetilde{\alpha})$-equivariant.
\end{exam}

The assumption that $\alpha$ and $\beta$ have different saturation properties removes the above pathological case. In a similar vein, we demand that $\beta$ is unsaturated so that it is in some sense close to the trivial action, as we know that Corollary \ref{cor:mainkjoins} applies when $\beta$ is trivial. The general case of Question \ref{ques:NCjoin}, however, remains open.

\section{Acknowledgments}

I am very grateful to Marc Rieffel, Alan Weinstein, and Guoliang Yu for hosting visits to University of California, Berkeley and Texas A\&M University, during which a fair portion of this work was completed, and to Piotr M. Hajac and Mariusz Tobolski for a very careful reading and thorough commentary. I would also like to thank my advisors John McCarthy and Xiang Tang at Washington University in St. Louis for their extensive support during my graduate study, and Orr Shalit and Baruch Solel for their support during my current postdoc at Technion-Israel Institute of Technology. Last, but certainly not least, the manuscript has been significantly improved by the comments of the referee.

\def\polhk#1{\setbox0=\hbox{#1}{\ooalign{\hidewidth
  \lower1.5ex\hbox{`}\hidewidth\crcr\unhbox0}}}
  \def\polhk#1{\setbox0=\hbox{#1}{\ooalign{\hidewidth
  \lower1.5ex\hbox{`}\hidewidth\crcr\unhbox0}}}
  \def\polhk#1{\setbox0=\hbox{#1}{\ooalign{\hidewidth
  \lower1.5ex\hbox{`}\hidewidth\crcr\unhbox0}}}
  \def\polhk#1{\setbox0=\hbox{#1}{\ooalign{\hidewidth
  \lower1.5ex\hbox{`}\hidewidth\crcr\unhbox0}}} \def\cprime{$'$}

\end{document}